\documentclass[a4paper,12pt]{article}

\usepackage[english]{babel}
\usepackage{amsmath,amsfonts,amssymb,amsthm,bbm}

\usepackage[top=2cm, bottom=2cm, left=2.5cm, right=2.5cm]{geometry}

\usepackage{indentfirst}
\usepackage[T1]{fontenc} 

\usepackage{graphicx}
\usepackage{float}

\usepackage{epsfig}
\usepackage{epstopdf}

\newtheorem{theorem}{Theorem}[section] 
\newtheorem{lemma}[theorem]{Lemma}     
\newtheorem{definition}[theorem]{Definition}
\newtheorem{corollary}[theorem]{Corollary}

\numberwithin{equation}{section}

\newcommand{\K}{\mathbb K}
\renewcommand{\H}{\mathbb H}
\newcommand{\N}{\mathbb N}

\newcommand{\R}{\mathbb R}

\newcommand{\Normal}{\mathcal{N}}
\newcommand{\Mean}{\mathbb E\,}
\DeclareMathOperator{\Var}{Var}
\DeclareMathOperator{\Cov}{Cov}

\newcommand{\Fi}{\mathbf \Phi}
\newcommand{\fii}{\mathbf \phi}
\newcommand{\B}{\mathbf \Psi}
\newcommand{\I}{\mathbf{Id}}
\newcommand{\x}{\mathbbm x}
\newcommand{\xS}{\mathbbm x^{S}}
\newcommand{\y}{\mathbbm y}

\newcommand{\0}{\mathbf 0}

\newcommand*{\ip}[1]{\left\langle{#1}\right\rangle} 
\newcommand*{\norm}[1]{\left\lVert{#1}\right\rVert} 
\newcommand*{\conj}[1]{\overline{#1}} 
\newcommand*{\herm}[1]{#1^*} 
\newcommand*{\abs}[1]{\left|{#1}\right|} 

\renewcommand{\i}{\mathbf i}
\renewcommand{\j}{\mathbf j}
\renewcommand{\k}{\mathbf k}

\DeclareMathOperator{\supp}{supp}

\newcommand{\Prob}{\mathbb P}
\newcommand{\sphere}{\mathcal{S}}

\begin{document}

\begin{center}
\Large
\textbf{Quaternion Gaussian matrices satisfy the RIP}\\
\end{center}


\begin{center}
\begin{tabular}{cc}
	\textbf{Agnieszka Bade\'nska\textsuperscript{1}} & \textbf{{\L}ukasz B{\l}aszczyk\textsuperscript{1,2}} \\
	badenska@mini.pw.edu.pl & l.blaszczyk@mini.pw.edu.pl 
\medskip \\
	\textsuperscript{1} Faculty of Mathematics  & \textsuperscript{2} Institute of Radioelectronics \\ and Information Science & and Multimedia Technology \\
	Warsaw University of Technology & Warsaw University of Technology \\
	ul. Koszykowa 75 & ul. Nowowiejska 15/19 \\
	00-662 Warszawa, Poland & 00-665 Warszawa, Poland
\end{tabular}
\end{center}


\bigskip
\normalsize
\begin{abstract}
We prove that quaternion Gaussian random matrices satisfy the~restricted isometry property (RIP) with overwhelming probability. We also explain why the~restricted isometry random variables (RIV) approach is not appropriate for drawing conclusions on restricted isometry constants.
\end{abstract}

\section{Introduction}

One of the~conditions which guarantees exact reconstruction of a~sparse signal (real, complex or quaternion) -- by $\ell_1$-norm minimization -- from a~few number of its linear measurements is that the~measurement matrix satisfies the~celebrated \textit{restricted isometry property} (RIP) with a~sufficiently small constant. The~notion of restricted isometry constants was introduced by Cand\`es and Tao in~\cite{ct} and repeatedly considered afterwards. The~concept was also generalized to quaternion signals \cite{bb1,bb2}.
\begin{definition}\label{qRIPdef}
Let $\Fi\in\H^{m\times n}$ and $s\in\{1,\ldots,n\}$. The~$s$-restricted isometry constant of~$\Fi$ is the~smallest number $\delta_s$ with the~property that
\begin{equation}\label{qRIP}
	 \left(1-\delta_s\right)\norm{\x}_2^2\leq\norm{\Fi\x}_2^2\leq\left(1+\delta_s\right)\norm{\x}_2^2 
\end{equation}
for all $s$-sparse quaternion vectors $\x\in\H^n$.
\end{definition}
\noindent Recall that we call a~vector (signal) $s$-sparse if it has at most $s$ nonzero coordinates and $\H$ denotes the~quaternion algebra.

It is known that e.g. real Gaussian and Bernoulli random matrices, also partial Discrete Fourier Transform matrices satisfy the~RIP (with overwhelming probability)~\cite{introCS}, however, until recently there were no known examples of quaternion matrices satisfying this condition. In~{\cite[Lemma~3.2]{bb1}} we proved that if a~real matrix $\Fi\in\R^{m\times n}$ satisfies the~inequalities~\eqref{qRIP} for real $s$-sparse vectors~$\x\in\R^n$, then it also satisfies it -- with the~same constant~$\delta_s$ -- for $s$-sparse quaternion vectors~$\x\in\H^n$. This was a~first step towards developing theoretical background of compressed sensing methods in the~quaternion algebra, since we also showed that it is possible to reconstruct sparse quaternion vectors from a~small number of their linear measurements if an~appropriate restricted isometry constant of a~\textit{real} measurement matrix $\Fi\in\R^{m\times n}$ is sufficiently small~{\cite[Corollary~5.1]{bb1}}.

Later we extended the above-mentioned result to the~full quaternion setting -- with the~expected assumption that a~quaternion measurement matrix $\Fi\in\H^{m\times n}$ satisfies the~RIP with a~sufficiently small constant~{\cite[Corollary~5.1]{bb2}}. The~question of existence of quaternion matrices satisfying the~RIP was, however, still open. 
Let us also mention a~very interesting recent result of N.~Gomes, S.~Hartmann and U.~K\"ahler concerning the~quaternion Fourier matrices -- arising in color representation of images~\cite{ghk}. They showed that with high probability such matrices allow a~sparse reconstruction (again by $\ell_1$-norm minimization) {\cite[Theorem~3.2]{ghk}}. Their proof is straightforward and does not use the~RIP notion, however, they also point out that it is not clear whether quaternion sampling matrices fulfill the~RIP condition.

It has been believed that quaternion Gaussian random matrices satisfy the~RIP and, therefore, they have been widely used in numerical experiments \cite{barthelemy2015,hawes2014,l1minqs} but there was a~lack of theoretical justification for this conviction. In this article we prove that this hypothesis is true, i.e. \textit{quaternion Gaussian random matrices satisfy the~RIP}, and we provide estimates on matrix sizes that guarantee the~RIP with overwhelming probability (Theorem~\ref{GaussQRIP}). The existence of quaternion matrices satisfying the~RIP, together with the~main results of~\cite{bb2}, constitute the~theoretical foundation of the~classical compressed sensing methods in the~quaternion algebra. 

\medskip
Apart from the aforementioned numerical experiments, our~main motivation was the~article~\cite{bb2}, which revealed the~need for quaternion matrices satisfying the~RIP. In our research we got interested in the~articles~\cite{jameslee,james1,james2}, the~authors of which claim they had found a~simpler proof of the~RIP for real and complex Gaussian random matrices. This brought us to study the~ratio random variable, so-called \textit{Rayleigh quotient}
$$ \mathcal{R}=\frac{\norm{\Fi\x}_2^2}{\norm{\x}_2^2} $$
for $\Fi\in\H^{m\times n}$ and non-zero, $s$-sparse vector $\x\in\H^n$, and eventually finding its distribution, i.e. $\Gamma(2m,2m)$, which does not depend on~$\x$, (Lemma~\ref{Rdistribution}). Its mean (expectation) equals~$1$ and its variance equals~$\frac{1}{2m}$.  Recall that in the~real case $\mathcal{R}$~has distribution $\Gamma\left(\frac{m}{2},\frac{m}{2}\right)$ with the~same mean but four times bigger variance~$\frac{2}{m}$~\cite{jameslee}. It clearly explains higher rate of correct sparse signals reconstructions in the~quaternion case in compare with the~real case of the~same size (see~\cite[Section 6]{bb2}).

Having proved Lemma~\ref{Rdistribution}, we were initially hoping to mimic the~approach from~\cite{jameslee,james1,james2} to obtain the~proof of the~RIP for quaternion Gaussian matrices, however, in our humble opinion it contains certain inaccuracies, which we want to briefly point out before proceeding further (we refer to~\cite{jameslee} for more details). The~authors call a~vector~$\x$ $s$-sparse if it has \textit{exactly} $s$ non-zero coordinates and introduce left and right $s$-\textit{restricted isometry constants} (RIC) as
\begin{equation}\label{RICs}
	 \delta_s^L=1-\min_{\#\supp\,\x=s}\frac{\norm{\Fi\x}_2^2}{\norm{\x}_2^2}, \qquad \delta_s^R=\max_{\#\supp\,\x=s}\frac{\norm{\Fi\x}_2^2}{\norm{\x}_2^2}-1,
\end{equation}
where $\#\supp\,\x$ denotes the~cardinality of the~support set of~$\x$ (i.e. the~number of its nonzero coordinates).
Note that it differs substantially from our approach from Definition~\ref{qRIPdef} ($\max(\delta_s^L,\delta_s^R)$ might be smaller than~$\delta_s$ and we loose monotonicity of~$\delta_s$ with respect to~$s$). Moreover, the~set of vectors with a~fixed cardinality of the~support does not have the~structure of a~module over the~ring~$\H$ anymore, hence the~above variables may not be well-defined (one might need to use $\inf/\sup$ instead).

After noticing that all ratio random variables $\frac{\norm{\Fi\x}_2^2}{\norm{\x}_2^2}$ have the~same distribution, which does not depend on the~choice of~$\x$, the~aforementioned authors only consider one vector~$\xS$ with $\supp\,\xS=S$ for every support set $S$ with $\#S=s$ and define the~following variable, called \textit{restricted isometry random variable} (RIV)
$$ \Delta_s^R=\max_{S:\#S=s}\left\{\frac{\norm{\Fi\xS}_2^2}{\norm{\xS}_2^2} \;\text{ for some }\xS \text{ with } \supp\,\xS=S \right\}-1 $$
($\Delta_s^L$ analogously), that is omitting $\max/\min$ over uncountable sets of vectors supported on appropriate~sets~$S$. That is true that for a~fixed vector and full ensemble of Gaussian matrices all Rayleigh quotients~$\mathcal{R}$ have the~same distribution, yet the~vector that realizes the~$\max/\min$ on a~fixed support set~$S$ depends on the~realization matrix in general, therefore the~distribution of $\max\left\{\frac{\norm{\Fi\x}_2^2}{\norm{\x}_2^2}:\supp\x=S\right\}$ very likely differs from the~distribution of a~single $\mathcal{R}$ for a~fixed vector. Therefore, one cannot draw any conclusions concerning~$\delta_s^L$, $\delta_s^R$ or~$\delta_s$ based on the~upper estimates of $\Delta_s^L$, $\Delta_s^R$ since obviously
$$ \Delta_s^L\leq\delta_s^L\leq\delta_s \quad\text{and}\quad \Delta_s^R\leq\delta_s^R\leq\delta_s. $$
Empirical distributions of~RICs and RIVs are presented in section~\ref{numexp} in~Fig.~\ref{fig:5_1} and~\ref{fig:5_2}.

Finally, the~authors of~\cite{jameslee} claim to have derived the~\textit{precise distribution functions} of $\Delta_s^R$ and $\Delta_s^L$ using, what they call, an '\textit{i.i.d. representation}', even though they admit that the~Rayleigh quotients for different supports are not independent if the~support sets intersect. In fact, they suggest a~lower estimate of the~cumulative distribution functions (CDF) of the~RIVs by the~CDF of a~maximum of ${n \choose s}$ (the~number of support sets) i.i.d. variables of the~same distribution as~$\mathcal{R}$ \cite[Appendix~E, Step~1]{jameslee}, however, their proof is very unsatisfactory due to the~lack of theoretical justification for the~presented inequalities. For the~case of two random variables what they really use is so-called \textit{positive dependence}, proved for certain distributions (including Gamma and Chi-square) in more general case, i.e.
$$ \Prob(X_1\in A, X_2\in A)\geq \Prob(X_1\in A)\cdot\Prob(X_2\in A) $$
for all measurable sets~$A\subset\R$ (see e.g. \cite{jensen}). For the~bigger number of variables with Gamma distribution this result was extended to a~certain extent \cite{royen} but -- to our best knowledge -- there is no proof for the~general case.

To sum up, it might be somewhat interesting to study the~RIVs in the~compressed sensing analysis, however, it brings no information on the~upper estimate of the~restricted isometry constants. Therefore, this approach is not useful for attempts of proving the~RIP. In fact, the~eigenvalue approach, i.e. analysis of moduli of eigenvalues of Hermitian matrices $\herm{\Fi_S}\Fi_S$, criticized by the~authors of~\cite{jameslee}, is the~right direction -- it is well known that the~restricted isometry constants can be expressed in terms of these eigenvalues (also in the~quaternion case -- cf.~{\cite[Lemma~2.1]{bb2}}). Currently however, we do not know distribution of the~eigenvalues of~$\herm{\Fi_S}\Fi_S$ in the~quaternion case. 

\medskip
That is why our proof of the~RIP of quaternion Gaussian random matrices refers to the~classical approach for real sub-Gaussian random matrices~{\cite[section~9.1]{introCS}}, which uses the~notion of sub-exponential (locally sub-Gaussian) random variables and a~$\gamma$-covering of a~unit ball. We are aware that our result (Theorem~\ref{GaussQRIP}) is not optimal, and in the~current form only applies to the~case of big~$n$ and very small~$s$, but it is a~first step towards further thorough analysis of the~quaternion random matrices and their applications in compressed sensing.

\medskip
The~article is organized as follows. First, we briefly recall basic facts about the~quaternion algebra and quaternion matrices. Third section is devoted to quaternion random variables and matrices. We define the~quaternion Gaussian random variable with mean zero and variance~$\sigma^2$, denoted by $X\sim\Normal_{\H}\left(0,\sigma^2\right)$, where in particular we always assume independence of its components -- note that this aspect was not clear in~\cite{l1minqs}. We also provide distribution of the~Rayleigh quotient~$\mathcal{R}$ for quaternion Gaussian random matrices and justify that it is sub-exponential with appropriate parameters. In the~fourth section we prove the~main result of the~article, i.e. the~RIP for quaternion Gaussian random matrices. Finally, we present outcomes of numerical simulations illustrating the~theoretical considerations and we conclude with a~short r\'esum\'e of the~obtained results and our further research perspectives.

\section{Quaternions}

Denote by~$\H$ the~algebra of (\textit{Hamilton} or \textit{real}) quaternions
$$ q=a+b\i+c\j+d\k, \quad \textrm{where}\quad a,b,c,d\in\R $$
endowed with the~standard norm
$$ |q|=\sqrt{q\conj{q}}=\sqrt{a^2+b^2+c^2+d^2}, $$
where $\conj{q}=a-b\i-c\j-d\k$ is the~conjugate of~$q$. Recall that multiplication is associative but in general not commutative in the~quaternion algebra and is defined by the~following rules
$$ \i^2=\j^2=\k^2=\i\j\k=-1 \quad\text{and}\quad \i\j=-\j\i=\k, \quad \j\k=-\k\j=\i, \quad \k\i=-\i\k=\j. $$
Multiplication is distributive with respect to addition and has a~neutral element $1\in\H$, hence $\H$ forms a~ring, which is usually called a~noncommutative field.

In this article we consider vectors with quaternion coordinates, i.e. elements of $\H^n$. Algebraically $\H^n$ is a~module over the~ring~$\H$, usually called the~\textit{quaternion vector space} (although it is not a~vector space, since $\H$ is not a~field). For any $n\in\N$ we consider the~Hermitian form $\ip{\cdot,\cdot}\colon\H^n\times\H^n\to\H$ with quaternion values
$$ \ip{\x,\y}=\herm{\y}\x=\sum_{i=1}^{n}\conj{y_i}x_i, \quad\textrm{for}\quad \x=(x_1,\ldots,x_n)^T,\; \y=(y_1,\ldots,y_n)^T\in\H^n, $$
where $\herm{\y}=\conj{\y}^T$, which satisfies axioms of an~inner product (in terms of the~right quaternion vector space, i.e. considering the~right scalar multiplication) and the~Cauchy-Schwarz inequality with respect to the~norm
$$ \norm{\x}_2=\sqrt{\ip{\x,\x}}=\sqrt{\sum_{i=1}^{n}|x_i|^2}, \quad \textrm{for any }\x=(x_1,\ldots,x_n)^T\in\H^n. $$
See {\cite[Section~2]{bb2}} for the~details.

We are particularly interested in matrices with quaternion entries with usual multiplication rules.
A~matrix $\Fi\in\H^{m\times n}$ defines a~$\H$-linear transformation $\Fi:\H^n\to\H^m$ (in~terms of the~right quaternion vector space) which acts by the~standard matrix-vector multiplication. By $\herm{\Fi}$ we denote the~adjoint matrix, i.e. $\herm{\Fi}=\conj{\Fi}^T$, which defines the~adjoint $\H$-linear transformation, i.e.
$$ \ip{\x,\herm{\Fi}\y}=\herm{\left(\herm{\Fi}\y\right)}\x=\herm{\y}\Fi\x=\ip{\Fi\x,\y} \quad \textrm{for} \quad \x\in\H^n,\; \y\in\H^m. $$
Recall also that a~linear transformation (matrix) $\mathbf{\Psi}\in\H^{n\times n}$ is called \textit{Hermitian} if $\herm{\mathbf{\Psi}}=\mathbf{\Psi}$. Obviously, $\herm{\Fi}\Fi$ is Hermitian for any $\Fi\in\H^{m\times n}$. 

Below we recall certain property of quaternion Hermitian matrices {\cite[Lemma~2.1]{bb2}}, which will be used later on.
\begin{lemma}\label{herm-norm}
Suppose $\mathbf{\Psi}\in\H^{n\times n}$ is Hermitian. Then
$$ \norm{\mathbf{\Psi}}_{2\to2}=\max_{\x\in\H^n,\norm{\x}_2=1}\left|\ip{\mathbf{\Psi}\x,\x}\right|=\max_{\x\in\H^n\setminus\{\0\}}\frac{\left|\ip{\mathbf{\Psi}\x,\x}\right|}{\norm{\x}_2^2}, $$
where $\norm{\cdot}_{2\to2}$ is the~standard operator norm in the~right quaternion vector space $\H^n$ endowed with the~norm $\norm{\cdot}_2$, i.e.
$$ \norm{\mathbf{\Psi}}_{2\to2}=\max_{\x\in\H^n\setminus\{\0\}}\frac{\norm{\mathbf{\Psi}\x}_2}{\norm{\x}_2}= \max_{\x\in\H^n,\norm{\x}_2=1}\norm{\mathbf{\Psi}\x}_2. $$
\end{lemma}

\section{Quaternion Gaussian random matrices}

For a~real random variable~$X$ we will denote its expectation (mean) by~$\Mean X$ and its variance by~$\Var X$. For Gamma distribution $\Gamma(\alpha,\beta)$ with shape parameter~$\alpha>0$ and rate parameter~$\beta>0$, i.e. random variable~$X$ with the~probability density function
$$ \gamma_{\alpha,\beta}(x) = \frac{\beta^{\alpha}}{\Gamma(\alpha)}x^{\alpha-1} e^{-\beta x}\quad \text{for }\,x\in(0,+\infty), $$
we have that
$$ \Mean X=\frac{\alpha}{\beta} \quad\text{and}\quad \Var X=\frac{\alpha}{\beta^2}.$$
Recall also that a~sum of squares of $k$ independent standard Gaussian random variables $\Normal(0,1)$ has Chi-square distribution with $k$ degrees of freedom, i.e. $\chi^2(k)=\Gamma\left(\frac{k}{2},\frac{1}{2}\right)$.

\medskip
Quaternion random variables have not been studied so far as thoroughly as their real or complex counterparts. However, during last two decades they attracted the~attention of researches both in theoretical and applied sciences \cite{lb,vc}.
Quaternion random variable $X$ is defined by four real random variables
$$ X = X_0 + X_1\i + X_2\j + X_3\k $$
and as such can be associated with the~four-dimensional real random vector $\left(X_1,X_2,X_3,X_4\right)$. There are several definitions of a~quaternion Gaussian random variable~\cite{vc}. The~most general (so-called \textit{R-Gaussian}) calls the~quaternion variable~$X$ Gaussian if $\left(X_1,X_2,X_3,X_4\right)$ is a~Gaussian random vector in~$\R^4$.

In this article we will only consider \textit{quaternion R-Gaussian random variables with independent components}. More precisely, we will assume that 
$$ X_i\sim\Normal\left(0,\frac{\sigma^2}{4}\right), \; i=1,\ldots,4, \quad\text{and}\quad X_i \text{ are pairwise independent.}$$
Such variables $X = X_0 + X_1\i + X_2\j + X_3\k$ will be called Gaussian with mean zero and variance~$\sigma^2$ and denoted by $X\sim\Normal_{\H}\left(0,\sigma^2\right)$. 

In what follows we will consider quaternion random matrices with independent entries sampled from quaternion Gaussian distribution, which has been defined above.
Let us emphasize once again that we will always assume \textit{independence of components} of quaternion Gaussian random variables.

\begin{lemma}\label{Rdistribution}
Let $\Fi=\left(\fii_{ij}\right)\in\H^{m\times n}$ be a~quaternion Gaussian random matrix whose entries are independent random variables with the~distribution $\Normal_{\H}\left(0,\frac{1}{m}\right)$ and let $0\neq\x\in\H^n$. Then the~random variable
$$ \mathcal{R}=\frac{\norm{\Fi\x}_2^2}{\norm{\x}_2^2} $$
has Gamma distribution $\Gamma(2m,2m)$ and it does not depend on~$\x$. In particular, $\Mean\mathcal{R}=1$ and $\Var\mathcal{R}=\frac{1}{2m}$.
\end{lemma}

\begin{proof}
Since $ \frac{\Fi\x}{\norm{\x}_2}=\Fi\left(\frac{\x}{\norm{\x}_2^2}\right)$, without loss of generality we can assume that $\norm{\x}_2=1$ and hence $\mathcal{R}=\norm{\Fi\x}_2^2$. Let us decompose the~matrix~$\Fi$ into its components:
\begin{align*}
	\Fi = \Fi_\mathbf{r} + \Fi_\i \i + \Fi_\j \j + \Fi_\k \k,\quad\text{where}\quad \Fi_{\mathbf{r}},\Fi_{\i},\Fi_{\j},\Fi_{\k}\in\R^{m\times n},
\end{align*} and analogously every matrix entry can be written as
\begin{align*}
	\fii_{ij} = \fii_{\mathbf{r},ij} + \fii_{\i,ij} \i + \fii_{\j,ij} \j + \fii_{\k,ij} \k \quad\text{with}\quad \fii_{e,ij}\in\R \quad\text{for}\quad e\in\{\mathbf{r},\i,\j,\k\}. 
\end{align*} 
We will use the~same indices to denote components of the~vector~$\x=(x_1,\ldots,x_n)^T$.

\medskip
Let $\Fi\x=\y=(y_1,\ldots,y_m)^T$, then
\begin{align*}
	y_k &= \sum\limits_{\ell=1}^n \fii_{k\ell} x_\ell
	=\sum\limits_{\ell=1}^n (\fii_{\mathbf{r},k\ell} + \fii_{\i,k\ell} \i + \fii_{\j,k\ell} \j + \fii_{\k,k\ell} \k)\cdot(x_{\mathbf{r},\ell} + x_{\i,\ell}\i + x_{\j,\ell}\j + x_{\k,\ell}\k) \displaybreak[1]\\
	&= \sum\limits_{\ell=1}^n (\fii_{\mathbf{r},k\ell}x_{\mathbf{r},\ell} - \fii_{\i,k\ell} x_{\i,\ell} - \fii_{\j,k\ell} x_{\j,\ell} - \fii_{\k,k\ell} x_{\k,\ell}) \displaybreak[3] \\
	&\quad + \sum\limits_{\ell=1}^n (\fii_{\i,k\ell}x_{\mathbf{r},\ell} + \fii_{\mathbf{r},k\ell} x_{\i,\ell} - \fii_{\k,k\ell} x_{\j,\ell} + \fii_{\j,k\ell} x_{\k,\ell})\i \displaybreak[3] \\
	&\quad + \sum\limits_{\ell=1}^n (\fii_{\j,k\ell}x_{\mathbf{r},\ell} + \fii_{\k,k\ell} x_{\i,\ell} + \fii_{\mathbf{r},k\ell} x_{\j,\ell} - \fii_{\i,k\ell} x_{\k,\ell})\j \displaybreak[3] \\
	&\quad + \sum\limits_{\ell=1}^n (\fii_{\k,k\ell}x_{\mathbf{r},\ell} - \fii_{\j,k\ell} x_{\i,\ell} + \fii_{\i,k\ell} x_{\j,\ell} + \fii_{\mathbf{r},k\ell} x_{\k,\ell})\k \displaybreak[3] \\
	&=: y_{\mathbf{r},k} + y_{\i,k}\i + y_{\j,k}\j + y_{\k,k}\k.
\end{align*} 
Recall that
$$ \fii_{e,ij}\sim\Normal\left(0,\frac{1}{4m}\right) \quad\text{for}\quad e\in\{\mathbf{r},\i,\j,\k\}, \quad i\in\{1,\ldots,m\}, \quad j\in\{1,\ldots,n\}. $$
Hence for each $e\in\{\mathbf{r},\i,\j,\k\}$ and every $k\in\{1,\ldots,m\}$, random variables $y_{e,k}$ are Gaussian (as linear combinations of Gaussian random variables) with $\Mean y_{e,k}=0$ and $\Var y_{e,k}=\frac{1}{4m}$ -- since all $\fii_{e,k\ell}$ are independent and
\begin{align*}
	\Var y_{\mathbf{r},k} &= \sum\limits_{\ell=1}^n (x_{\mathbf{r},\ell}^2\Var\fii_{\mathbf{r},k\ell} + x_{\i,\ell}^2\Var\fii_{\i,k\ell} + x_{\j,\ell}^2\Var\fii_{\j,k\ell}+ x_{\k,\ell}^2\Var\fii_{\k,k\ell}) \\
	&= \frac{1}{4m}\sum\limits_{\ell=1}^n (x_{\mathbf{r},\ell}^2 + x_{\i,\ell}^2 + x_{\j,\ell}^2 + x_{\k,\ell}^2) = \frac{\norm{\x}_2^2}{4m} = \frac{1}{4m}.
\end{align*}
For the~remaining components analogously.

Independence of the~variables $\fii_{e,k\ell}$ implies also independence of $y_{e,k}$ and $y_{e,\ell}$ for every fixed $e\in\{\mathbf{r},\i,\j,\k\}$ and all pairs $k,\ell\in\{1,\ldots,m\}$, $k\neq\ell$. 
In order to verify independence of $y_{\mathbf{r},k}$ and $y_{\i,k}$, $k\in\{1,\ldots,m\}$, observe that
\begin{align*}
	\Cov(y_{\mathbf{r},k},y_{\i,k}) &= \Mean (y_{\mathbf{r},k}\cdot y_{\i,k}) - \Mean y_{\mathbf{r},k}\cdot \Mean y_{\i,k} \\
	&= \Mean\left[\left(\sum\limits_{\ell=1}^n (\fii_{\mathbf{r},k\ell}x_{\mathbf{r},\ell} - \fii_{\i,k\ell} x_{\i,\ell} - \fii_{\j,k\ell} x_{\j,\ell} - \fii_{\k,k\ell} x_{\k,\ell})\right)\right.\\
	&\qquad\quad \left.\cdot\left(\sum\limits_{p=1}^n (\fii_{\i,kp}x_{\mathbf{r},p} + \fii_{\mathbf{r},kp} x_{\i,p} - \fii_{\k,kp} x_{\j,p} + \fii_{\j,kp} x_{\k,p})\right)\right] \\
	&= \sum\limits_{\ell=1}^n \Mean\big((\fii_{\mathbf{r},k\ell}x_{\mathbf{r},\ell} - \fii_{\i,k\ell} x_{\i,\ell} - \fii_{\j,k\ell} x_{\j,\ell} - \fii_{\k,k\ell} x_{\k,\ell}) \\
	&\qquad\quad \cdot(\fii_{\i,k\ell}x_{\mathbf{r},\ell} + \fii_{\mathbf{r},k\ell} x_{\i,\ell} - \fii_{\k,k\ell} x_{\j,\ell} + \fii_{\j,k\ell} x_{\k,\ell})\big) \\
	&= \sum\limits_{\ell=1}^n (x_{\mathbf{r},\ell}x_{\i,\ell}\Mean\fii_{\mathbf{r},k\ell}^2 - x_{\i,\ell}x_{\mathbf{r},\ell}\Mean\fii_{\i,k\ell}^2 - x_{\j,\ell}x_{\k,\ell}\Mean\fii_{\j,k\ell}^2 + 	
		x_{\k,\ell}x_{\j,\ell}\Mean\fii_{\k,k\ell}^2 ) \\
	&= \frac{1}{4m} \sum\limits_{\ell=1}^n (x_{\mathbf{r},\ell}x_{\i,\ell} - x_{\i,\ell}x_{\mathbf{r},\ell}  - x_{\j,\ell}x_{\k,\ell}  + x_{\k,\ell}x_{\j,\ell} ) = 0,
\end{align*} 
since $\Mean\fii_{e,k\ell}^2=\Var\fii_{e,k\ell} = \frac{1}{4m}$. The~same way one argues that covariance of the~remaining pairs is also zero. Recall that real Gaussian random vectors have independent components if and only if their covariance equals zero. 

We get, therefore, that all $y_{e,k}$, $e\in\{\mathbf{r},\i,\j,\k\}$ and $k\in\{1,\ldots,m\}$, are independent (real) random variables with distribution $\Normal\left(0,\frac{1}{4m}\right)$ and hence $\sqrt{4m}\,y_{e,k}\sim\Normal(0,1)$. This implies that
\begin{align*}
	4m\mathcal{R} = \norm{\sqrt{4m}\y}_2^2 = \sum\limits_{\ell=1}^m \left(\left(\sqrt{4m}y_{\mathbf{r},\ell}\right)^2 + \left(\sqrt{4m}y_{\i,\ell}\right)^2 + \left(\sqrt{4m}y_{\j,\ell}\right)^2 + 		
									\left(\sqrt{4m}y_{\k,\ell}\right)^2\right)
\end{align*} 
is a~sum of $4m$ squares of independent standard Gaussian random variables $\Normal(0,1)$ and consequently, $4m\mathcal{R}$ has Chi-square distribution $\chi^2(4m)=\Gamma(2m,\frac{1}{2})$. We conclude that $\mathcal{R}$~has distribution $\Gamma(2m,2m)$, independently of~$\x$. This random variable has mean $\frac{2m}{2m}=1$ and variance $\frac{2m}{(2m)^2}=\frac{1}{2m}$.
\end{proof}

As we previously remarked, in the~real case $\mathcal{R}$ has distribution~$\Gamma\left(\frac{m}{2},\frac{m}{2}\right)$, that is with four times bigger variance~\cite{jameslee}. It explains the~aforementioned better results of quaternion sparse vectors reconstruction in compare with the~real case~{\cite[Fig.~1 and Fig.~2~(a)]{bb2}}, since a~quaternion Gaussian random matrix statistically has smaller restricted isometry constant than its real counterpart.

\medskip
Let us now proceed with the~tools needed for the~proof of the~RIP. Recall that a~real random variable $X$ is called \textit{sub-exponential} (\textit{locally sub-Gaussian})~\cite{Wainwright} if there exist $\sigma^2>0$ and $\delta>0$ such that
\begin{align*}
	\Mean\big(\mathrm{e}^{t(X-\Mean X)}\big) \leq \exp\left(\frac{\sigma^2 t^2}{2}\right)\quad \text{for }\abs{t}\leq\delta.
\end{align*} 
We will denote it $X\sim\mathrm{SubExp}(\sigma^2,\delta)$. Equivalently, one may write that 
$$ M(t)=\Mean\left(\mathrm{e}^{tX}\right) \leq \exp\left(t\,\Mean X+\frac{\sigma^2 t^2}{2}\right)\quad \text{for }\abs{t}\leq\delta,$$
where $M(t)$ is the~moment generating function. 

\medskip
It is known that a~random variable with Gamma distribution is sub-exponential with any $\sigma^2>\Var X$ for some $\delta>0$~\cite{Wainwright}. Below we recall a~simple proof of this fact, with $\sigma^2$ and $\delta$ chosen for our purposes, in which we shall use the~following form of the~moment generating function of the~Gamma distribution $\Gamma(\alpha,\beta)$
$$ M(t)=\frac{1}{\left(1-\frac{t}{\beta}\right)^{\alpha}}\quad\text{for}\; t<\beta. $$

\begin{lemma}\label{GammaSubExp}
Let $X$ has distribution $\Gamma(\alpha,\beta)$. Then $X\sim\mathrm{SubExp}\left(\frac{5}{2}\Var X,\frac{\beta}{5}\right)$.
\end{lemma}

\begin{proof}
Indeed, take $|t|\leq\frac{\beta}{5}$, which equivalently means that $\frac{|t|}{\beta}\leq\frac{1}{5}$. Then
\begin{align*}
	\Mean\big(e^{t(X-\frac{\alpha}{\beta})}\big) &= \frac{1}{\left(1-\frac{t}{\beta}\right)^{\alpha}} \cdot \mathrm{e}^{-t\frac{\alpha}{\beta}}
																								= \left(1+\frac{t}{\beta}+\frac{(\frac{t}{\beta})^2}{1-\frac{t}{\beta}}\right)^{\alpha} \cdot \mathrm{e}^{-t\frac{\alpha}{\beta}} \\
		&\!\!\!\!\!\overset{1-\frac{t}{\beta}\geq\frac{4}{5}}{\leq} \left(1+\frac{t}{\beta}+\frac{5}{4}\left(\frac{t}{\beta}\right)^2\right)^{\alpha} \cdot \mathrm{e}^{-t\frac{\alpha}{\beta}} \\
		&\leq \exp\left(\alpha\cdot\left(\frac{t}{\beta} + \frac{5}{4}\left(\frac{t}{\beta}\right)^2\right)\right) \cdot \exp\left(-t\frac{\alpha}{\beta}\right) \\
		&= \exp\left(\frac{1}{2}\cdot\frac{5}{2}\frac{\alpha}{\beta^2}\cdot t^2\right),
\end{align*} 
where we used well known estimation $1+x\leq\mathrm{e}^x$ for $x\in\R$. Hence $\sigma^2 =\frac{5}{2}\frac{\alpha}{\beta^2}=\frac{5}{2}\Var X$ and $\delta=\frac{\beta}{5}$.
\end{proof}

In what follows we will also use the~following known fact~\cite{Wainwright}: if $X\sim\mathrm{SubExp}(\sigma^2,\delta)$, then
$$ 	\Prob\left(\abs{X-\Mean X}\geq t\right) \leq 2\exp\left(-\frac{t^2}{2\sigma^2}\right)\quad \text{for}\quad 0\leq t\leq \sigma^2\delta. $$

\begin{corollary}
	The random variable $\mathcal{R}\sim\Gamma(2m,2m)$ from Lemma~\ref{Rdistribution} is sub-exponential with parameters $\sigma^2=\frac{5}{2}\cdot\frac{1}{2m}=\frac{5}{4m}$ and $\delta=\frac{2m}{5}$. Hence
$$ \Prob\left(\abs{\mathcal{R}-1}\geq t\right) \leq 2\exp\left(-\frac{t^2}{2\sigma^2}\right)\quad \text{for}\quad 0\leq t\leq \frac{1}{2}, $$
and therefore
\begin{equation}\label{Restimation}
 \forall_{0\neq\x\in\H^n}\quad \forall_{0\leq t\leq \frac{1}{2}}\quad \Prob\left(\abs{\norm{\Fi\x}_2^2-\norm{\x}_2^2}\geq t\norm{\x}_2^2\right) \leq 2\exp\left(-\frac{t^2}{2\sigma^2}\right). 
\end{equation}
\end{corollary}

\section{Proof of the RIP}\label{RIPproof}

As it was stated in the~introduction, we say that a~matrix $\Fi\in\H^{m\times n}$ satisfies the~$s$-restricted isometry property (for quaternion vectors) with a~constant $\delta_s\geq0$ if the~inequalities~\eqref{qRIP} holds for all $s$-sparse quaternion vectors $\x\in\H^n$. The~smallest number $\delta_s\geq0$ with this property is called the $s$-restricted isometry constant. Without loss of generality one can only consider $s$-sparse unit vectors, i.e. $\norm{\x}_2=1$. Moreover, in~{\cite[Lemma~3.2]{bb2}} we proved that the~$s$-restricted isometry constant $\delta_s$ of $\Fi\in\H^{m\times n}$ equivalently equals
\begin{equation}\label{sRIPconstant}
	\delta_s=\max_{S\subset\{1,\ldots,n\},\#S\leq s}\norm{\herm{\Fi_S}\Fi_S-\I}_{2\to2}, 
\end{equation}
where $\Fi_S$ is the~submatrix of~$\Fi$ consisting of columns with indices from $S\subset\{1,\ldots,n\}$.

\medskip
We begin with the~following result in which we fix the~support $S\subset\{1,\ldots,n\}$.
\begin{lemma}\label{suppRIP}
Let $\Fi\in\H^{m\times n}$ be a~quaternion Gaussian matrix whose entries $\fii_{ij}$ are independent quaternion random variables with distribution $\mathcal{N}_{\H}(0,\frac{1}{m})$. Moreover, let the~set $S\subset\{1,\ldots,n\}$ be such that $\# S=s\leq n$. For any $\delta\in\left(0,\frac{1}{\sqrt3}\right)$ and $\varepsilon\in(0,1)$, if
$$ m \geq \frac{10}{3}\delta^{-2}\left(14s + \ln\left(\frac{2}{\varepsilon}\right)\right),$$
then with probability at least $1-\varepsilon$ we have that
$$ \norm{\Fi_S^*\Fi_S-\mathbf{Id}}_{2\to 2}<\delta. $$
\end{lemma}

\begin{proof}
Fix a~set $S\subset\{1,\ldots,n\}$ with $\# S=s$ and denote
$$ \mathcal{A}_S=\left\{\x\in\H^n:\quad \supp\,\x\subset S \quad \wedge \quad \norm{\x}_2=1\right\}. $$
This set can be associated with the~unit sphere ${\sphere}^{4s-1}$ in~$\R^{4s}$.

\medskip
Take a~number $0<\gamma<\frac{1}{2}$ (the~exact value of~$\gamma$ will be specified later). By {\cite[Proposition~C.3]{introCS}} there exists a~$\gamma$-covering~$\mathcal{A}_{\gamma}$ of~$\mathcal{A}_S$ such that
$$ \# \mathcal{A}_{\gamma}\leq \left(1+\frac{2}{\gamma}\right)^{4s}. $$
For any~$0\leq\tilde{\delta}\leq\frac{1}{2}$, using~\eqref{Restimation}, we get that
\begin{align*}
\Prob\left(\exists_{\y\in\mathcal{A}_{\gamma}}\quad \abs{\norm{\Fi\y}_2^2-\norm{\y}_2^2}\geq\tilde{\delta}\norm{\y}_2^2\right) &= 
					\Prob\left( \bigcup\limits_{\y\in\mathcal{A}_{\gamma}} \left\{ \abs{\norm{\Fi\y}_2^2-\norm{\y}_2^2}\geq\tilde{\delta}\norm{\y}_2^2\right\}\right) \\
	&\leq \sum\limits_{\y\in\mathcal{A}_{\gamma}} \Prob\left(\abs{\norm{\Fi\y}_2^2-\norm{\y}_2^2}\geq\tilde{\delta}\norm{\y}_2^2\right) \\
	&\leq \#\mathcal{A}_{\gamma} \cdot 2\exp\left(-\frac{\tilde{\delta}^2}{2\sigma^2}\right) \\
	&\leq 2\left(1+\frac{2}{\gamma}\right)^{4s}\exp\left(-\frac{\tilde{\delta}^2}{2\sigma^2}\right)
\end{align*}
with $\sigma^2=\frac{5}{4m}$. This implies that
$$ \Prob\left(\forall_{\y\in\mathcal{A}_{\gamma}}\quad \abs{\norm{\Fi\y}_2^2-\norm{\y}_2^2}<\tilde{\delta}\norm{\y}_2^2\right)\geq 
				1 - 2\left(1+\frac{2}{\gamma}\right)^{4s}\exp\left(-\frac{\tilde{\delta}^2}{2\sigma^2}\right). $$

Since $\mathcal{A}_{\gamma}\subset\mathcal{A}_S$, we obviously have that
\begin{equation}\label{goodmatrix}
	\forall_{\y\in\mathcal{A}_{\gamma}}\quad \abs{\norm{\Fi\y}_2^2-\norm{\y}_2^2}<\tilde{\delta}\norm{\y}_2^2 \quad\Leftrightarrow\quad
			\forall_{\y\in\mathcal{A}_{\gamma}}\quad \abs{\norm{\Fi_S\y_S}_2^2-\norm{\y_S}_2^2}<\tilde{\delta}\norm{\y_S}_2^2 
\end{equation}
and
$$ \abs{\norm{\Fi_S\y_S}_2^2-\norm{\y_S}_2^2} = \abs{\ip{(\Fi_S^*\Fi_S-\mathbf{Id})\y_S,\y_S}}.$$
For a~matrix~$\Fi$ satisfying \eqref{goodmatrix} denote $\B=\Fi_S^*\Fi_S-\mathbf{Id}$. Since all vectors in~$\mathcal{A}_{\gamma}$ are unit and supported on~$S$, \eqref{goodmatrix} implies that
$$ \forall_{\y\in\mathcal{A}_{\gamma}}\quad \left|\ip{\B\y_S,\y_S}\right|<\tilde{\delta}\norm{\y_S}_2^2=\tilde{\delta}\norm{\y}_2^2=\tilde{\delta}. $$
By the~definition of a~$\gamma$-covering, for every $\x\in\mathcal{A}_S$ there is some $\y\in\mathcal{A}_{\gamma}$ such that $\norm{\x-\y}_2\leq\gamma<\frac{1}{2}$. Since both $\x$ and $\y$ are unit and supported on~$S$, using properties of the~Hermitian norm and quaternion Cauchy-Schwarz inequality, we get that
\begin{align*}
	\abs{\ip{\B\x_S,\x_S}} & = \abs{\ip{\B\y_S,\y_S} + \ip{\B\x_S,\x_S-\y_S} + \ip{\B(\x_S-\y_S),\y_S}}\\
		&\leq \abs{\ip{\B\y_S,\y_S}} + \norm{\B}_{2\to 2}\norm{\x}_2\norm{\x-\y}_2 + \norm{\B}_{2\to 2}\norm{\x-\y}_2\norm{\y}_2 \\
		&< \tilde{\delta} + 2\gamma\norm{\B}_{2\to 2}.
\end{align*}
In view of Lemma~\ref{herm-norm}, since $\B$ is Hermitian, taking supremum over all $\x\in\mathcal{A}_S$, we obtain
$$ \norm{\B}_{2\to 2} < \tilde{\delta} + 2\gamma\norm{\B}_{2\to 2} \quad\Rightarrow\quad \norm{\B}_{2\to 2} < \frac{\tilde{\delta}}{1-2\gamma}. $$
Denoting $\delta=\frac{\tilde{\delta}}{1-2\gamma}\leq\frac{1}{2}\cdot\frac{1}{1-2\gamma}$ we get that
\begin{align*}
	\Prob\left(\norm{\Fi_S^*\Fi_S-\mathbf{Id}}_{2\to2}<\delta\right) &\geq
	\Prob\left(\forall_{\y\in\mathcal{A}_{\gamma}}\quad \abs{\norm{\Fi\y}_2^2-\norm{\y}_2^2}<\tilde{\delta}\norm{\y}_2^2\right) \\
	& \geq 1 - 2\left(1+\frac{2}{\gamma}\right)^{4s}\exp\left(-\frac{\tilde{\delta}^2}{2\sigma^2}\right) \\
	& = 1 - 2\left(1+\frac{2}{\gamma}\right)^{4s}\exp\left(-\frac{2}{5}\delta^2(1-2\gamma)^2\,m\right).
\end{align*}

It implies that if
$$ m \geq \frac{5}{2}\cdot\frac{\delta^{-2}}{(1-2\gamma)^2}\left(4s\cdot\ln\left(1+\frac{2}{\gamma}\right)+\ln\left(\frac{2}{\varepsilon}\right)\right), $$
then
\begin{equation}\label{big-probab}
	\Prob\left(\norm{\Fi_S^*\Fi_S-\mathbf{Id}}_{2\to2}<\delta\right) \geq 1 - \varepsilon.
\end{equation}
Taking $\gamma=\frac{2}{e^{7/2}-1}\approx6.23\cdot 10^{-2}$, for which $\frac{1}{(1-2\gamma)^2}\leq \frac{4}{3}$ and $\ln\left(1+\frac{2}{\gamma}\right)=\frac{7}{2}$, we finally obtain that for any positive $\delta\leq\frac{1}{2}\cdot\frac{2}{\sqrt{3}}=\frac{1}{\sqrt{3}}$, if
$$ m \geq \frac{10}{3}\delta^{-2}\left(14s + \ln\left(\frac{2}{\varepsilon}\right)\right), $$
then \eqref{big-probab} holds, which concludes the~proof.
\end{proof}

\medskip
We are ready to prove the~main result.
\begin{theorem}\label{GaussQRIP}
	Let $\Fi\in\H^{m\times n}$ be a~quaternion Gaussian matrix whose entries $\fii_{ij}$ are independent quaternion random variables with distribution $\mathcal{N}_{\H}(0,\frac{1}{m})$. For any $\delta\in\left(0,\frac{1}{\sqrt3}\right)$ and $\varepsilon\in(0,1)$, if
$$ m \geq \frac{10}{3}\delta^{-2}\left(15s + \ln\left(\frac{2}{\varepsilon}\right) + s\ln\left(\frac{n}{s}\right)\right),$$
then with probability at least $1-\varepsilon$ the~$s$-restricted isometry constant~$\delta_s$ of~$\Fi$ satisfies $\delta_s<\delta$.
\end{theorem}

\begin{proof}	
Using \eqref{sRIPconstant}, the~proof of Lemma~\ref{suppRIP} and well known estimates of the~Newton's symbol we get that
\begin{align*}
	\Prob(\delta_s\geq\delta) &= \Prob\left(\max\limits_{S\colon \# S=s} \norm{\Fi_S^*\Fi_S-\mathbf{Id}}_{2\to2}\geq\delta\right) \\
		&= \Prob\left( \exists_{S\colon \# S=s}\quad \norm{\Fi_S^*\Fi_S-\mathbf{Id}}_{2\to2}\geq\delta\right) \\
		&= \Prob\left(\bigcup\limits_{S\colon \# S=s} \left\{\norm{\Fi_S^*\Fi_S-\mathbf{Id}}_{2\to2}\geq\delta\right\} \right) \\
		&\leq \sum\limits_{S\colon \# S=s} \Prob\left(\norm{\Fi_S^*\Fi_S-\mathbf{Id}}_{2\to2}\geq\delta\right) \\
		&\leq {n \choose s}\cdot2\left(1+\frac{2}{\gamma}\right)^{4s}\exp\left(-\frac{2}{5}\delta^2(1-2\gamma)^2\,m\right) \\
		&\leq 2\left(\frac{\mathrm{e} n}{s}\right)^s \left(1+\frac{2}{\gamma}\right)^{4s}\exp\left(-\frac{2}{5}\delta^2(1-2\gamma)^2\,m\right).
\end{align*}
Therefore if
$$ m \geq \frac{5}{2}\cdot\frac{\delta^{-2}}{(1-2\gamma)^2}\left(s\ln\left(\frac{\mathrm{e}n}{s}\right)+4s\cdot\ln\left(1+\frac{2}{\gamma}\right)+\ln\left(\frac{2}{\varepsilon}\right)\right), $$
then $\Prob(\delta_s<\delta)\geq 1-\varepsilon$. Taking again $\gamma=\frac{2}{e^{7/2}-1}$ we get the~thesis.
\end{proof}

\section{Numerical experiment}\label{numexp}

For the~sake of illustrating our discussion regarding the~RICs and the~RIVs~\cite{jameslee,james1,james2}, we performed a~numerical experiment of evaluating empirical distributions of the~random variables $\delta_s^R$, $\delta_s^L$, $\Delta_s^R$ and $\Delta_s^L$ for the~case of real Gaussian matrices $\Fi\in\R^{m\times n}$. The~experiment was carried out in MATLAB R2016a on a~standard PC machine, with Intel(R) Core(TM) i7-4790 CPU (3.60GHz), 16GB RAM and with Microsoft Windows 10 Pro.

In order to efficiently derive empirical distributions, we set the~size of Gaussian matrices $\Fi\in\R^{m\times n}$ to $m=64$, $n=8$, and the~sparsity of unit ($\norm{\x}_2=1$) vectors $\x\in\R^{n}$ to $s=5$. This choice of $n$ and $s$ gave us ${n \choose s}=56$ different support sets $S\subset\{1,\ldots,n\}$ with $\# S=s$, and for each support~$S$ we drew $10^3$ unit vectors $\x\in\R^n$ with~$\supp\x=S$. Hence, for every realization of the~random variable~$\Fi$, we took maximum over the~set consisting of $56\cdot 10^3$ elements to estimate distributions of $\delta_s^R$ and $\delta_s^L$. Estimating  distributions of $\Delta_s^R$ and $\Delta_s^L$ required only one vector~$\x$ with $\supp\x=S$ for each support set~$S$ . The~matrix sample consisted of $10^5$ realizations of the~random variable~$\Fi$.

Results of this experiment, i.e. empirical probability and cumulative distribution functions, are shown in~Fig.\ref{fig:5_1} and~\ref{fig:5_2}. As one can see, knowledge of any~upper bound of the~variables $\Delta_s^R$ and $\Delta_s^L$ (with certain probability) brings no estimate on the~upper bound of the~RICs~$\delta_s^R$ and~$\delta_s^L$.

\begin{figure}[ht]
	\centerline{
	\epsfig{file=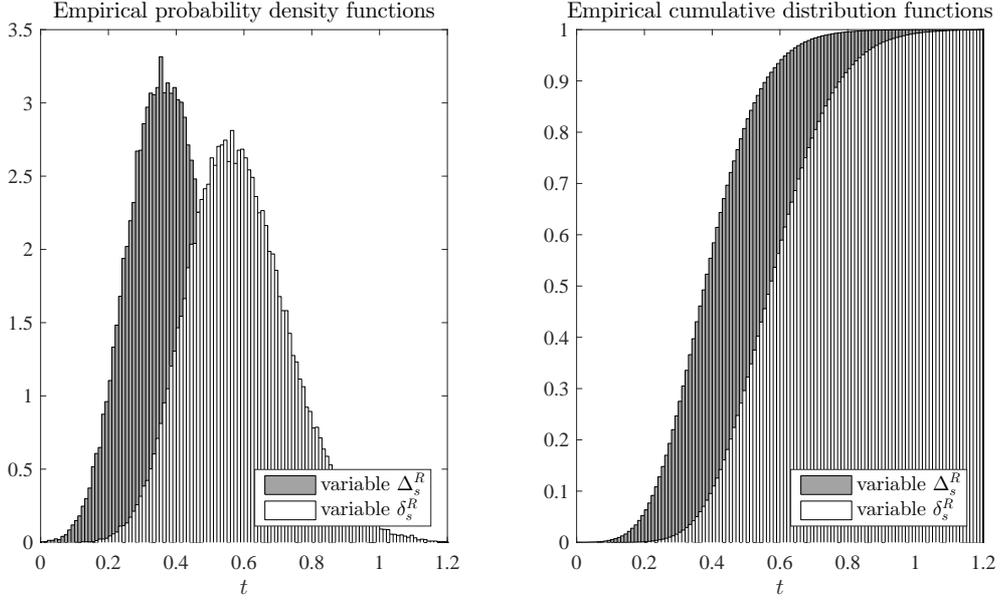,width=\textwidth}
	}
	\caption{Empirical distributions of $\delta_s^R$ and $\Delta_s^R$ for Gaussian random matrices $\Fi\in\R^{m\times n}$, $n=8$, $m=64$, $s=5$.}
	\label{fig:5_1}
\end{figure}

\begin{figure}[!h]
	\centerline{
	\epsfig{file=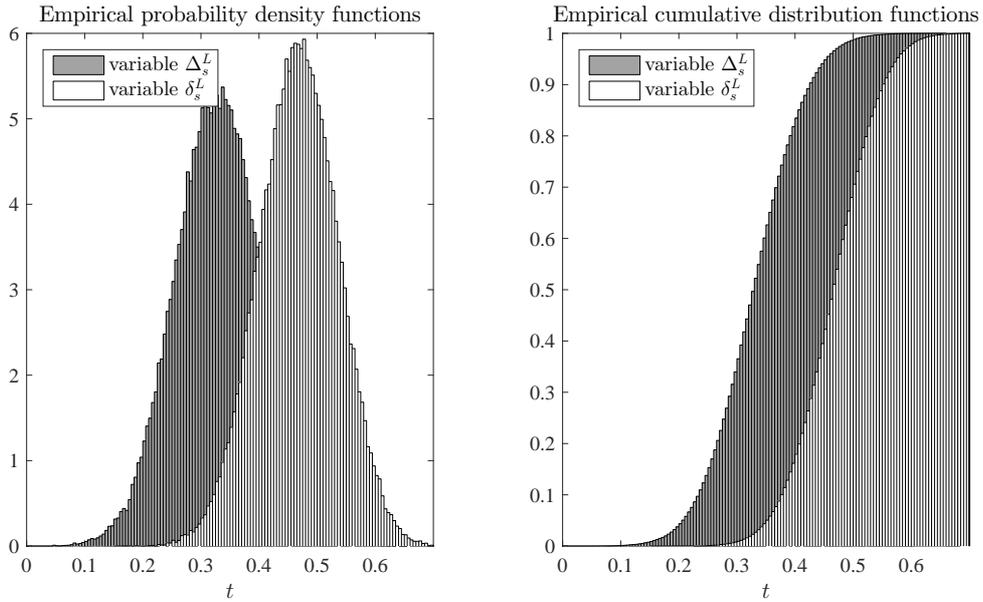,width=\textwidth}
	}
	\caption{Empirical distributions of $\delta_s^L$ and $\Delta_s^L$ for Gaussian random matrices $\Fi\in\R^{m\times n}$, $n=8$, $m=64$, $s=5$.}
	\label{fig:5_2}
\end{figure}

\medskip
In the~second experiment we derived empirical distribution of the~restricted isometry constant~$\delta_s$ for quaternion Gaussian random matrices and compared it with the~case of real Gaussian random matrices. We carried out the~simulations for matrices $\Fi\in\K^{m\times n}$, where $\K=\H$ or $\K=\R$, with parameters $n=256$, $m=64$, and for unit $s$-sparse vectors $\x\in\K^{n}$ such that $s=\frac{m}{2}=32$. 

Recall that the~$s$-restricted isometry constant can be equivalently defined as
$$ \delta_s = \max \abs{\norm{\Fi\x}_2^2-1}, $$
where the maximum is taken over the~set of all $s$-sparse vectors $\x\in\K^n$ such that $\norm{\x}_2=1$. To estimate the~cumulative distribution function of the~random variable~$\delta_s$ for each random matrix~$\Fi$ we drew $10^5$ $s$-sparse unit vectors and took maximum over this set. The~matrix sample consisted of $10^4$ realizations of the~random variable~$\Fi$ in both cases. Entries of the~matrices were independently sampled from the~normal distribution with mean zero and variance~$\frac{1}{m}$, i.e. $\Normal\left(0,\frac{1}{m}\right)$ in the~real case and $\Normal_{\H}\left(0,\frac{1}{m}\right)$ in the~quaternion case.

Results of this experiment -- for real and quaternion random matrices -- are shown in Fig.~\ref{fig:5_3} and~Fig.~\ref{fig:5_4} respectively. In view of the~previous theoretical considerations, we can see that (with certain probability) the~$s$-restricted isometry constant of a~quaternion Gaussian random matrix is small. It proves, in particular, that there exist quaternion matrices satisfying the~RIP. Moreover, notice that the~$s$-restricted isometry constant of a~quaternion random matrix $\Fi\in\H^{m\times n}$ is statistically smaller than its real counterpart for a~matrix $\Fi\in\R^{m\times n}$ of the~same size. As we previously mentioned, we think that the~reason for this phenomenon is smaller variance of the~Rayleigh random variable in the~quaternion case (Lemma~\ref{Rdistribution}).

\begin{figure}[!h]
	\centerline{
	\epsfig{file=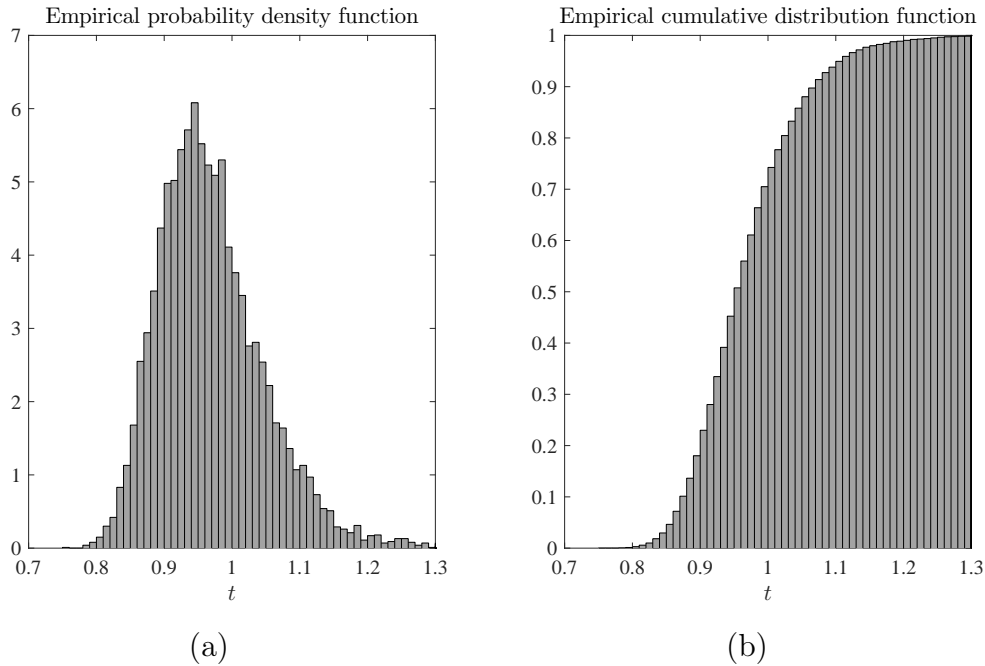,width=\textwidth}
	}
	\begin{minipage}{0.09\textwidth}\mbox{} \end{minipage}
	\begin{minipage}{0.36\textwidth} \centerline{(a)} \end{minipage}
	\begin{minipage}{0.07\textwidth}\mbox{} \end{minipage}
	\begin{minipage}{0.35\textwidth} \centerline{(b)}\end{minipage}
	\begin{minipage}{0.06\textwidth}\mbox{} \end{minipage}

	\caption{Empirical distributions of $\delta_s$ for Gaussian random matrices $\Fi\in\R^{m\times n}$, $n=256$, $m=64$, $s=32$.}
	\label{fig:5_3}
\end{figure}

\begin{figure}[!h]
	\centerline{
	\epsfig{file=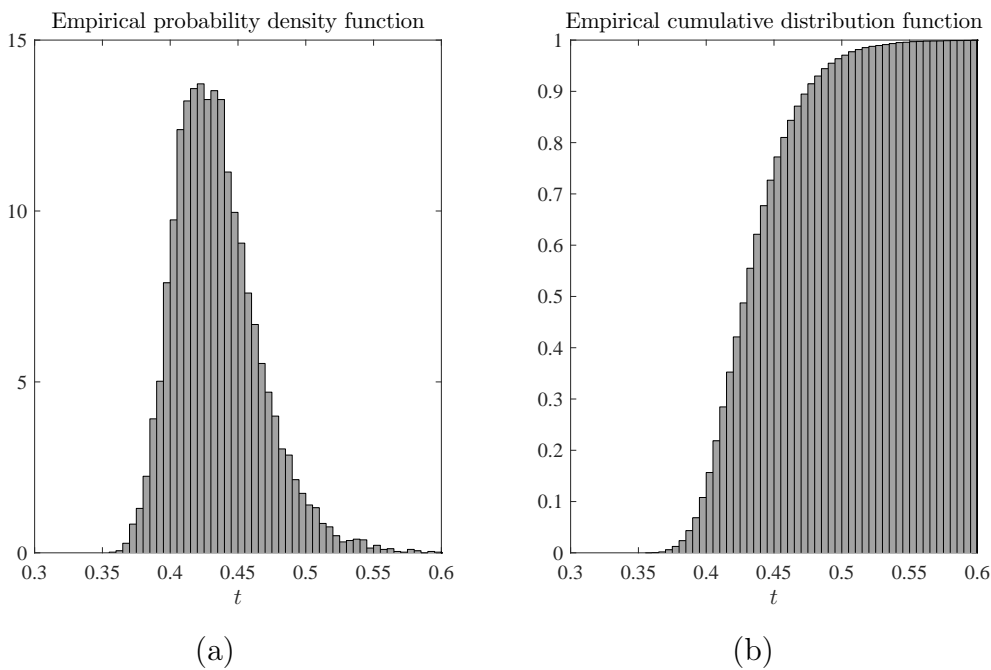,width=\textwidth}
	}
	\begin{minipage}{0.09\textwidth}\mbox{} \end{minipage}
	\begin{minipage}{0.36\textwidth} \centerline{(a)} \end{minipage}
	\begin{minipage}{0.07\textwidth}\mbox{} \end{minipage}
	\begin{minipage}{0.35\textwidth} \centerline{(b)}\end{minipage}
	\begin{minipage}{0.06\textwidth}\mbox{} \end{minipage}

	\caption{Empirical distributions of $\delta_s$ for Gaussian random matrices $\Fi\in\H^{m\times n}$, $n=256$, $m=64$, $s=32$.}
	\label{fig:5_4}
\end{figure}

Finally, one might be interested in comparing these results with the~formulation of~Theorem~\ref{GaussQRIP}. We should remember, however, that the~estimate in this theorem is significant only for very large $n$ (more than few thousand) and $s\ll n$ (in practice, less than a~dozen) and in our experiment, due to computational reasons, $n=256$ and $s=32$. As we also commented, we are aware that this~result is not sharp and the~numerical experiment additionally suggests that it can be improved in future.

\section{Conclusion}

This article brings positive answer to the~question about existence of quaternion matrices satisfying the~RIP. We confirm that -- as it was conjectured -- restricted isometry constants of Gaussian quaternion matrices are small with big probability (and typically smaller than their real counterparts). Together with our previous result, in which we proved that quaternion measurement matrices with small RIP constants allow exact reconstruction of sparse quaternion vectors, it explains success of compressed sensing based experiments in the~quaternion algebra and brings hope for their wider applications.

The~main result,  however, in the~current form is not sharp. One of the~reasons is that we used techniques previously applied to the~case of real subgaussian random matrices. It would be interesting in the~future to study distribution of singular values of quaternion Gaussian matrices in order to improve our result. Other direction of the~further research might be search for other quaternion matrices satisfying the~RIP, in particular, those more desired from the~application point of view, e.g. partial quaternion DFT matrix.

\bigskip\bigskip
\noindent
\textbf{Acknowledgments}

The~research was supported in part by WUT grant No. 504/02499/1120. 
The~work conducted by the~second author
was supported by a scholarship from Fundacja Wspierania Rozwoju Radiokomunikacji i Technik Multimedialnych.

\end{document}